\documentclass[12pt]{amsart}
\usepackage{amscd}
\usepackage{amsmath}
\usepackage{amssymb}
\usepackage{units}
\usepackage[all]{xy}
\def\frk{\frak}               

\def\Phi{{\frk n}}
\def\Phi{{\frk N}}
\def\opn#1#2{\def#1{\operatorname{#2}}} 
\opn\chara{char} \opn\length{\ell} \opn\pd{pd} \opn\rk{rk}
\opn\projdim{proj\,dim} \opn\injdim{inj\,dim} \opn\rank{rank}
\opn\depth{depth} \opn\sdepth{sdepth} \opn\fdepth{fdepth}
\opn\grade{grade} \opn\height{height} \opn\embdim{emb\,dim}
\opn\codim{codim}  \opn\min{min} \opn\max{max}

\opn\Tr{Tr} \opn\bigrank{big\,rank}
\opn\superheight{superheight}\opn\lcm{lcm}
\opn\trdeg{tr\,deg}
\opn\reg{reg} \opn\lreg{lreg} \opn\ini{in} \opn\lpd{lpd}
\opn\size{size}
\opn\div{div} \opn\Div{Div} \opn\cl{cl} \opn\Cl{Cl}
\opn\Spec{Spec} \opn\Supp{Supp} \opn\supp{supp} \opn\Sing{Sing}
\opn\Ass{Ass} \opn\Min{Min}
\opn\Ann{Ann} \opn\Rad{Rad} \opn\Soc{Soc}
\opn\Im{Im} \opn\Ker{Ker} \opn\Coker{Coker} \opn\Am{Am}
\opn\Hom{Hom} \opn\Tor{Tor} \opn\Ext{Ext} \opn\End{End}
\opn\Aut{Aut} \opn\id{id}  \opn\deg{deg}

\opn\nat{nat}
\opn\pff{pf}
\opn\Pf{Pf} \opn\GL{GL} \opn\SL{SL} \opn\mod{mod} \opn\ord{ord}
\opn\Gin{Gin} \opn\Hilb{Hilb}
\opn\aff{aff} \opn\con{conv} \opn\relint{relint} \opn\st{st}
\opn\lk{lk} \opn\cn{cn} \opn\core{core} \opn\vol{vol}
\opn\link{link} \opn\star{star}
\opn\gr{gr}

\def\pot#1#2{#1[\kern-0.28ex[#2]\kern-0.28ex]}

\opn\dirlim{\underrightarrow{\lim}}
\opn\inivlim{\underleftarrow{\lim}}
\let\to=\rightarrow

\def\Implies{\ifmmode\Longrightarrow \else
        \unskip${}\Longrightarrow{}$\ignorespaces\fi}
\def\implies{\ifmmode\Rightarrow \else
        \unskip${}\Rightarrow{}$\ignorespaces\fi}
\def\iff{\ifmmode\Longleftrightarrow \else
        \unskip${}\Longleftrightarrow{}$\ignorespaces\fi}

\let\:=\colon
\newtheorem{Theorem}{Theorem}[]
\newtheorem{Lemma}[Theorem]{Lemma}
\newtheorem{Corollary}[Theorem]{Corollary}

\newtheorem{Remark}[Theorem]{Remark}

\newtheorem{Conjecture}[Theorem]{Conjecture}

\let\epsilon\varepsilon
\let\phi=\varphi
\let\kappa=\varkappa
\def\qed{\ifhmode\textqed\fi
      \ifmmode\ifinner\quad\qedsymbol\else\dispqed\fi\fi}
\def\textqed{\unskip\nobreak\penalty50
       \hskip2em\hbox{}\nobreak\hfil\qedsymbol
       \parfillskip=0pt \finalhyphendemerits=0}
\def\dispqed{\rlap{\qquad\qedsymbol}}

\opn\dis{dis}
\def\pnt{{\raise0.5mm\hbox{\large\bf.}}}

\opn\Lex{Lex}

\begin{document}
\title{  Around General Neron Desingularization}

\author{ Dorin Popescu }
\thanks{The  support from the the project  ID-PCE-2011-1023, granted by the Romanian National Authority for Scientific Research, CNCS - UEFISCDI   is gratefully acknowledged. The author thanks  CIRM, Luminy who provided excellent conditions and stimulative atmosphere in the main stage of our work.}

\address{Dorin Popescu, Simion Stoilow Institute of Mathematics , Research unit 5,
University of Bucharest, P.O.Box 1-764, Bucharest 014700, Romania}
\email{dorin.popescu@imar.ro}

\maketitle

\begin{abstract} It gives  new forms  of General Neron Desingularization and new applications.

 \noindent
  {\it Key words } : Smooth morphisms,  regular morphisms, Bass-Quillen Conjecture, arcs.\\
 {\it 2010 Mathematics Subject Classification: Primary 13B40, Secondary 14B25,13H05,13J15.}
\end{abstract}

\vskip 0.5 cm

\section*{Introduction}

M.  Artin \cite{A} proved  that  the convergent power series rings
over $\bf C$ has the so called today the Artin approximation property. Using Artin's proof, Ploski \cite{Pl} stated a theorem in the idea of Neron desingularization \cite{N}. This was the starting point of our General Neron Desingularization (see  Theorem \ref{gnd}). Later other types of such theorems were published. One was  conjectured by Artin \cite{A3} and solved positively in \cite{CP} (see  Theorem \ref{cp}) using  an idea of \cite[Theorem 2.3]{AD}. Another one similarly to  the so called the strong Artin approximation property  \cite{P2} (see Theorem \ref{p2}), which has an interesting consequence (see Corollary \ref{corm}) studied in a special case in Section 2.
Computing  General Neron Desingularization with  SINGULAR (see \cite{Sing}) it was necessary to use  in \cite{AP}  maps defined by formal power series but not in an  explicit way. This inspires us Theorem \ref{arc}.

 Let $(A,m)$ be a discrete valuation ring, $B$ a finite type $A$-algebra, $\hat A$ its completion and $v:B\to A/m^{2c+1}$ an $A$-morphism, for a certain  $c\in \bf N$ depending on $B$. Then  there exists an $A$-morphism $v':B\to \hat A$ such that $v'\equiv v\ \mbox{modulo}\ m$ using  \cite[Theorem 18]{AP} - a theorem similar to the  Greenberg's form of the strong Artin approximation. Our Theorem \ref{arc1} shows that the set of all $A$-morphisms $B\to \hat A$ lifting $v$ is in bijection with an affine space over $\hat A$. This could be useful in the theory of  arcs  (see Corollary \ref{arc2}).
 The paper ends with an application of the General Neron Desingularization to an extension of the Bass-Quillen Conjecture. More precisely,  if $R$ is for example  an equicharacteristic regular local ring and  $I\subset R[X]$, $X=(X_1,\ldots,X_r)$ is a monomial ideal, then any finitely generated projective $R[X]/I$-module is free (see Theorem \ref{m1}).

\section{A version of General Neron Desingularization in the idea of Ploski}

A ring morphism $u:A\to A'$ has  {\em regular fibers} if for all prime ideals $P\in \Spec A$ the ring $A'/PA'$ is a regular  ring, i.e. its localizations are regular local rings. It has {\em geometrically regular fibers}  if for all prime ideals $P\in \Spec A$ and all finite field extensions $K$ of the fraction field of $A/P$ the ring  $K\otimes_{A/P} A'/PA'$ is regular. If $A\supset {\bf Q}$ then regular fibers of $u$ are geometrically regular. We call $u$ {\em regular} if it is flat and its fibers are geometrically regular. A Henselian local ring $(A,m)$ is {\em excellent} if it is Noetherian and the completion map $A\to {\hat A}$ is regular. A regular morphism is {\em smooth} if it is finitely presented and it is {\em essentially smooth} if it is a localization of a finitely presented morphism.

\begin{Theorem} (Ploski \cite{Pl}) \label{pl} Let ${\bf C}\{x\}$, $x=(x_1,\ldots,x_n)$, $f=(f_1,\ldots,f_s)$ be some convergent power series from ${\bf C}\{x,Y\}$, $Y=(Y_1,\ldots,Y_N)$ and $\hat y\in {\bf C}[[x]]^N$ with $\hat y(0)=0$ be a solution of $f=0$. Then the map $v:B={\bf C}\{x,Y\}/(f)\to {\bf C}[[x]]$ given by $Y\to \hat y$ factors through an $A$-algebra of type $B'={\bf C}\{x,Z\}$ for some variables $Z=(Z_1,\ldots,Z_s)$, that is $v$ is a composite map $B\to B'\to {\bf C}[[x]]$.
\end{Theorem}
This result shows a particular case of  the following theorem in Krull dimension $>1$, the case when $A,A'$  are DVR being given by the so called N\'eron $p$-desingularization \cite{N}, \cite{A1}.

\begin{Theorem} (General Neron Desingularization, Popescu \cite{P3}, \cite{P''}, Andre \cite{An}, Swan \cite{S}, Spivakovski \cite{Sp})\label{gnd}  Let $u:A\to A'$ be a  regular morphism of Noetherian rings and $B$ a finite type $A$-algebra. Then  any $A$-morphism $v:B\to A'$   factors through a smooth $A$-algebra $C$, that is $v$ is a composite $A$-morphism $B\to C\to A'$.
\end{Theorem}

The following theorem is a consequence of Theorem \ref{gnd} in the idea of Ploski (perhaps it  is  already known but we couldn't find a precise reference).

\begin{Theorem}\label{pp} Let $(A,m)$ be an excellent Henselian local ring, $\hat A$ its completion,  $B$ a finite type $A$-algebra and $v:B\to \hat A$ an $A$-morphism.  Then  $v$  factors through an $A$-algebra of type $ A[Z]^h$ for some variables $Z=(Z_1,\ldots,Z_s)$, where $A[Z]^h$ is the Henselization of $A[Z]_{(m,Z)}$.
\end{Theorem}
\begin{proof} By Theorem \ref{gnd} we see that $v$ factors  through a smooth $A$-algebra $B'$, let us say $v$ is the composite map $B\to B'\xrightarrow{v'} \hat A$.
Using the local structure of smooth algebras given by Grothendieck (see \cite[Theorem 2.5]{S}) we may assume that   $B'_{v'^{-1}(m\hat A)}$ is a localization of a smooth $A$-algebra of type $(A[Z,T]/(g))_{g'h}$, where $Z=(Z_1,\ldots,Z_s)$, $g'=\partial g/\partial T$. Then there exists $h\in A[Z,T]$ such that $v'$ factors through $C=(A[Z,T]/(g))_{g'h}$  let us say $v'$ is the composite map $B'\to C\xrightarrow{w} \hat A$.

 Choose $z_0\in A^N$, $t_0\in A$, such that $(z_0,t_0)\equiv (\hat z,\hat t)$ modulo $m\hat A$. Changing $(Z,T)$ by $(Z-z_0,T-t_o)$ in $C$ we may suppose that $w(Z),w(T)$ are in $m\hat A$.
 Then $w$ extends to a map $w':C'=(A[Z]^h[T]/(g))_{g'h}\to \hat A$ and we have $C'\cong A[Z]^h$ since $C'$ is an etale neighborhood of  $A[Z]^h$.
\hfill\ \end{proof}

A Noetherian local ring $(A,m)$ has the {\em property of Artin approximation} if the solutions in $A$ of  every  system of polynomial equations $f$ over $A$ is dense with respect of the $m$-adic topology in the set of solutions of $f$ in the completion $\hat A$ of $A$.
In fact $(A,m)$ has the  property of Artin approximation if and only if   every  system of polynomial equations $f$ over $A$ has a solution in $A$ if it has one in $\hat A$.
The following theorem was conjectured by M. Artin in \cite{A2}.

\begin{Theorem} (Popescu \cite{P3}, \cite{P1}) \label{main} An excellent Henselian local ring has the property of  Artin approximation.
\end{Theorem}

This theorem follows easily from Theorem \ref{gnd} using the Implicit Function Theorem. But it is much easier to apply Theorem \ref{pp}. Indeed, let $(A,m)$ be  an excellent Henselian local ring,  $f=(f_1,\ldots,f_r)$  some polynomials from $A[Y]$, $Y=(Y_1,\ldots,Y_N)$ and $\hat y\in \hat A^N$ a solution of $f=0$. We will show that $f$ has a solution in $A$. Set $B=A[Y]/(f)$ and $v:B\to \hat A$ the map given by $Y\to \hat y$. By Theorem \ref{pp} $v$ factors through $A[Z]^h$ for some $Z=(Z_1,\ldots,Z_s)$, that is $v$ is the composite map $B\xrightarrow{t} A[Z]^h\to \hat A$. Let $\alpha:A[Z]^h\to A[[Z]]$ be the canonical map and set $y=(\alpha t)(Y+(f))$. Then $y(0)$ is a solution of $f$ in $A$.

An extended   form of Theorem \ref{gnd} is the following theorem which is a positive answer  to a conjecture of M. Artin \cite{A3}.

\begin{Theorem} (Cipu-Popescu \cite{CP}) \label{cp} Let $u:A\to A'$ be a  regular morphism of Noetherian rings, $B$ a finite type $A$-algebra, $v:B\to A'$  an $A$-morphism and $D\subset \Spec B$ the open smooth locus of $B$ over $A$. Then there exist    a smooth $A$-algebra $C$ and two $A$-morphisms $t:B\to C$, $w:C\to A'$ such that $v=wt$ and $C$ is smooth over $B$ at ${t^*}^{-1}(D)$, $t^*:\Spec C\to \Spec B$ being induced by $t$.
\end{Theorem}

 Another type of General Neron Desingularization is the following theorem.

\begin{Theorem} (Popescu \cite{P2})\label{p2} Let $(A,m)$ be a Noetherian local ring with the completion map $A\to \hat A$ regular. Then for every finite type $A$-algebra $B$ there exists a function $\lambda:\bf N\to \bf N$ such that for every positive integer $c$ and every morphism $v:B\to A/m^{\lambda(c)}$ there exists a smooth $A$-algebra $C$ and two $A$-algebra morphisms $t:B\to C$, $w:C\to A/m^c$ such that   $wt$ is the composite map $B\xrightarrow{v} A/m^{\lambda(c)}\to A/m^c$.
\end{Theorem}

\begin{Remark}{\em The proof from \cite{P2}  is not constructive and so it is not clear that $\lambda$ can be computed. On the other hand $\hat A$ has the so called the  strong Artin approximation property (see \cite{A1}, \cite{PP}, \cite{K}, \cite{P3}, \cite{P1}). Suppose that $B=A[Y]/(f)$, $Y=(Y_1,\ldots,Y_n)$, $f=(f_1,\ldots,f_r)\in A[Y]^r$ and let $\nu$ be  the Artin function  associated to $f$ considered over $\hat A$. Next corollary relates somehow $\lambda $ with $\nu$. }
\end{Remark}

\begin{Corollary} \label{corm} Let $(A,m) $ be a Noetherian local ring with the completion map $A\to \hat A$ regular, $c\in \bf N$, $B$ a finite type $A$-algebra  and $\lambda$ its function defined by Theorem \ref{p2}. Let $g:B\to {\hat A}/m^{\lambda(c)}\hat A$ be an $A$-morphism. Then  there exists a   $B$-algebra $C$, which is smooth over $A$, and  an $A$-morphism $w:C\to \hat A$ such that the composite map $B\to C\xrightarrow{w} \hat A$ coincides with $g$ modulo $m^c\hat A$. In particular, $\lambda\geq \nu$, the last map being the Artin function over $\hat A$ associated to the system of polynomials $f$ defining $B$. Moreover,    for any $A$-morphism $v:B\to \hat A$ congruent  with $g$ modulo $m^{\lambda(c)}\hat A$,  there exists an $A$-morphism $w:C\to \hat A$ such that the composite map $B\to C\xrightarrow{w} \hat A$ coincides with $v$ modulo $m^c\hat A$.
\end{Corollary}
\begin{proof} Let $\lambda$ be given by Theorem \ref{p2}. Then  there exists a $B$-algebra $C$ which is smooth over $A$ and an $A$-morphism  ${\tilde w}:C\to A/m^c$ such that the composite map $B\to C\xrightarrow{\tilde w} {\hat A}/m^c{\hat A}$ coincides with $g$ modulo $m^c$. By the Implicit Function Theorem $\tilde w$ can be lifted to an $A$-morphism $w:C\to \hat A$. Therefore,  $g$  coincides with the composite map $B\to C\xrightarrow{w} \hat A$ modulo $m^c\hat A$.
\hfill\ \end{proof}
\begin{Remark} {\em It will be nice  to get above that  $v$ coincides with the composite map  $B\to C\xrightarrow{w} A'$ because then  $C$ could be chosen   independently  on $v$ but depending  on $g$ (clearly, $C$ can be replaced for example by a polynomial algebra over $C$ and so it is not unique). This is possible  in a special case (see Theorem \ref{arc}).}
\end{Remark}

\section{A special case of Corollary \ref{corm}}

Let  $A$ be a discrete valuation ring, $x$ a local parameter of $A$,  $A'=\hat A$ its completion and  $B=A[Y]/I$, $Y=(Y_1,\ldots,Y_n)$ a finite type $A$-algebra.
  If $f=(f_1,\ldots,f_r)$, $r\leq n$ is a system of polynomials from $I$ then we consider a $r\times r$-minor $M$ of the Jacobian matrix $(\partial f_i/\partial Y_j)$.  Let  $c\in \bf N$. Suppose that there exist  an $A$-morphism $v:B\to A'/(x^{2c+1})$   and  $N\in ((f):I)$ such that
 $v(NM)\not\in (x)^c/(x^{2c+1})$, where for simplicity we write $v(NM)$ instead $v(NM+I)$. We may assume that $M=\det((\partial f_i/\partial Y_j)_{i,j\in [r]})$.

\begin{Theorem}\label{arc} Then there exists a $B$-algebra $C$ which is smooth over $A$ such that every $A$-morphism $v':B\to A'$ with $v'\equiv v \ \mbox{modulo}\ x^{2c+1}$ (that is $v'(Y)\equiv v(Y) \ \mbox{modulo}\\
 x^{2c+1}$) factors through $C$.
\end{Theorem}
\begin{proof} We follow somehow the proof of \cite[Theorem 3]{AP} given in a more general assumptions.
Since $A/(x^{2c+1})\cong A'/(x^{2c+1})$ we may choose $y'\in A^n$, such that $v(Y)=y' + (x^{2c+1})$.
Set $P=NM$ and $d=P(y')$. Multiplying $N$ by a certain power of $x$ we may suppose that $dA=x^cA$.

 Let  $H$ be the $n\times n$-matrix obtained adding down to $(\partial f/\partial Y)$ as a border the block $(0|\mbox{Id}_{n-r})$ (we assume as above that $M$ is given on the first $r$ columns of the Jacobian matrix). Let $G'$ be the adjoint matrix of $H$ and $G=NG'$. We have
$$GH=HG=NM \mbox{Id}_n=P\mbox{Id}_n$$
and so
$$d\mbox{Id}_n=P(y')\mbox{Id}_n=G(y')H(y').$$

 Let
 $$h=Y-y'-dG(y'))T,$$
 where  $T=(T_1,\ldots,T_n)$ are new variables.  Since
$$Y-y'\equiv dG(y')T\ \mbox{modulo}\ h$$
and
$$f(Y)-f(y')\equiv \sum_j\partial f/\partial Y_j(y') (Y_j-y'_j)$$
modulo higher order terms in $Y_j-y'_j$ by Taylor's formula we see that  we have
$$f(Y)-f(y')\equiv  \sum_jd\partial f/\partial Y_j(y') G_j(y')T_j+d^2Q=$$
$$dP(y')T+d^2Q=d^2(T+Q)\
\mbox{modulo}\ h,$$
 where $Q\in T^2 A[T]^r$. This is because $(\partial f/\partial Y)G=(P\mbox{Id}_r|0)$.  We have $f(y')=d^2a$ for some $a\in xA^r$. Set
$g_i=a_i+T_i+Q_i$, $i\in [r]$ and  $E=A[Y,T]/(I,g,h)$. Clearly, it holds $d^2g\subset (f,h)$ and $(f)\subset (h,g)$ .

Note that $U=(A[T]/(g))_{s}$ is smooth for some $s\in 1+(T)$ because the $r\times r$-minor of the Jacobian matrix $(\partial g/\partial T)$ given by the first $r$ columns has the form $\det(\mbox{Id}_r+(\partial Q_i/\partial T_j)_{i,j\in [r]})$. In particular, $U$ is flat over $A$ and so $d$ is regular in $U$.
We claim that $E_{ss'}\cong U_{s'}$ for some $s'\in 1+(x,T)$. It will be enough to show that $I\subset (h,g)A[Y,T]_{ss'}$. We have $PI\subset (f)\subset (h,g)$ and so $P(y'+dL)I\subset (h,g)$ where $L=G(y')T$. But $P(y'+dL)$ has the form $ds'$ for some $s'\in 1+(d,T)$. It follows that $s'I\subset (h,g)A[Y,T]_{s}$ since $d$ is regular in $U$ and so $I\subset  (h,g)A[Y,T]_{ss'}$.
Thus  $C=E_{ss'}$ is a $B$-algebra  smooth over $A$.

Remains to see that an arbitrary $A$-morphism $v':B\to A'$ with $v'\equiv v\ \mbox{modulo} \ x^{2c+1}$  factors through $C$. We have $v'(Y)\equiv v(Y)\equiv y'\
\mbox{modulo}\ x^{2c+1}$ and so there exists $\epsilon\in xA'^n$ such that $v'(Y)-y'=d^2\epsilon$.
 Then $t:=H(y')\epsilon\in xA'^n$
satisfies
$$G(y')t=P(y')\epsilon=d\epsilon$$
 and so
 $$v'(Y)- y'=dG(y')t,$$
that is $h(v'(Y),t)=0$.
Note that $d^2g(t)\in (h(v'(Y)),t), f(v'(Y)))=(0)$
and it follows that $g(t)=0$ since $A'$ is a domain. Thus $v'$ factors through $E$, that  is $v'$ is a composite map $B\to E\xrightarrow{\alpha} A'$, where $\alpha$ is a $B$-morphism given by $T\to t$.
 As $\alpha(s)\equiv 1$ modulo $x$ and $\alpha(s')\equiv 1$ modulo $(x,t)$, $t\in xA'$ we see that $\alpha(s),\alpha(s')$ are invertible because  $A'$ is local  and so $\alpha$ (thus $v'$) factors through the standard smooth $A$-algebra $C$.
\hfill\ \end{proof}

\begin{Remark} \label{gr} {\em In  Theorem \ref{arc} always we get an $A$-morphism $v'':B\to A'$ such that $v''\equiv v \ \mbox{modulo}\ m$ using a theorem of type Greenberg (see \cite{Gr} and \cite[Theorem 18]{AP}) but $\{v'\in \Hom_A(B,A'):v'\equiv v \ \mbox{modulo}\ x^{2c+1}\}$ could be empty (here $\Hom_A(B,A')$ denotes the set of all $A$-morphisms $B\to A'$).}
\end{Remark}

\begin{Corollary}\label{c1} In the assumptions and notations of the above theorem, let $\rho:B\to C$ be the structural algebra map. Then $\rho$ induces a bijection $\rho^*$ between $\{w\in \Hom_A(C,A'):w\rho\equiv v \ \mbox{modulo}\ x^{2c+1}\}$ and
 $\{v'\in \Hom_A(B,A'):v'\equiv v \ \mbox{modulo}\ x^{2c+1}\}$ given by $\rho^*(w)=w\rho$.
 \end{Corollary}
 \begin{proof} By the above theorem $\rho^*$ is surjective. Let $w,w'\in \Hom_A(C,A')$ be such that $w\rho=w'\rho\equiv v\ \mbox{modulo}\ x^{2c+1}$.
Since  $H(y')(Y-y')\equiv x^{2c}T\ \mbox{modulo}\ h$ by construction of $E$ we get $x^{2c}(w(T)-w'(T))=0$ and so $w|_E=w'|_E$ because $A'$ is a domain. It follows $w=w'$.
\hfill\ \end{proof}

\begin{Corollary}\label{c2} In the assumptions and notations of the above corollary, suppose that there exists an $A$-morphism ${\tilde v}:B\to A'$
 with ${\tilde v}\equiv v\ \mbox{modulo}\ x^{2c+1}$. Then
 there exists an unique $A$-morphism ${\tilde w}:C\to A'$ such that ${\tilde w}\rho={\tilde v}$.
  \end{Corollary}
  For the proof take
  ${\tilde w}={\rho^*}^{-1}(\tilde v)$, where $\rho^*$ is defined in the above corollary.

 By  a Grothendieck's theorem \cite[Theorem 2.5]{S}, $C$ can be chosen of the form $(A[Z,V]/(g))_{g'h}$, where $Z=(Z_1,\ldots,Z_s)$ and $g'=\partial g/\partial V$.
\begin{Lemma}\label{lem} There  exists a canonical bijection
     $$ A'^s\to  \{w'\in \Hom_A(C,A'):w'\equiv {\tilde w} \ \mbox{modulo}\ x^{2c+1}\}.$$
   \end{Lemma}
\begin{proof} Let $w'\in \Hom_A(C,A')$ be with $w'\equiv {\tilde w}\ \mbox{modulo}\ x^{2c+1}$.
We have
$$g(w'(Z),{\tilde w}(V))\equiv g({\tilde w}(Z),{\tilde w}(V))=0\ \mbox{modulo}\ x^{2c+1},$$
$$g'(w'(Z),{\tilde w}(V))\equiv g'({\tilde w}(Z),{\tilde w}(V))\not \equiv 0\ \mbox{modulo}\ x,$$
$$h(w'(Z),{\tilde w}(V))\equiv h({\tilde w}(Z),{\tilde w}(V))\not \equiv 0\ \mbox{modulo}\ x.$$
Thus $g(w'(Z),V)=0$ has an unique solution  (namely $w'(V)$) in ${\tilde w}(V)+x^{2c+1}A'$ by the Implicit Function Theorem.
It follows that the restriction
$$\{w'\in \Hom_A(C,A'):w'\equiv {\tilde w} \ \mbox{modulo}\ x^{2c+1}\}\to $$
$$\{w''\in \Hom_A(A[Z],A'):w''\equiv {\tilde w}|_{A[Z]} \ \mbox{modulo}\ x^{2c+1}\}$$
is bijective.

On the other hand, the map
$$\{w''\in \Hom_A(A[Z],A'):w''\equiv {\tilde w}|_{A[Z]} \ \mbox{modulo}\ x^{2c+1}\}\to {\tilde w}(Z)+ x^{2c+1}A'^s$$
given by $w''\to w''(Z)$ is a bijection too. It follows that there exists a canonical bijection between
$\{w'\in \Hom_A(C,A'):w'\equiv {\tilde w} \ \mbox{modulo}\ x^{2c+1}\}$ and  ${\tilde w}(Z)+ x^{2c+1}A'^s$, the last one being in bijection with $A'^s$.
\hfill\ \end{proof}

\begin{Theorem} \label{arc1} In the assumptions and notations of Corollary \ref{c1} there exists a canonical bijection
$$A'^s\to \{v'\in \Hom_A(B,A'):v'\equiv v \ \mbox{modulo}\ x^{2c+1}\}$$
 for some $s\in \bf N$.
 \end{Theorem}
For the proof apply Corollary \ref{c1} and the above lemma.

 Let $k$ be a field and $F$  a finite type $k$-algebra, let us say $F=k[U]/J$, $U=(U_1,\ldots,U_n)$. An arc $\Spec k[[x]]\to \Spec F$ (see \cite{LJ}) is given by a $k$-morphism $F\to A'=k[[x]]$.
Assume that $H_{F/k}\not =0$ (this happens for example when $F$ is reduced and $k$ is perfect). Set $A=k[x]_{(x)}$, $B=A\otimes_kF$. Let
   $f=(f_1,\ldots,f_r)$, $r\leq n$ be a system of polynomials from $J$ and $M$ a $r\times r$-minor  of the Jacobian matrix $(\partial f_i/\partial U_j)$.  Let  $c\in \bf N$. Assume that there exist  an $A$-morphism $g:F\to A'/(x^{2c+1})$   and  $N\in ((f):J)$ such that
 $g(NM)\not\in (x)^c/(x^{2c+1})$.
 Note that $A\otimes_k-$ induces a bijection $\Hom_k(F,A')\to \Hom_A(B,A')$ by adjunction.

 \begin{Corollary}\label{arc2} The set $\{g'\in \Hom_k(F,A'):g'\equiv g \ \mbox{modulo}\ x^{2c+1}\}$ is in bijection with an affine space $A'^s$ over $A'$ for some $s\in \bf N$.
 \end{Corollary}

\section{An application  of General Neron Desingularization to an extension of Bass-Quillen Conjecture}

Let $R[T]$, $T=(T_1,\ldots,T_n)$ be a polynomial algebra in $T$ over a Noetherian ring $R$. A finitely generated projective $R[T]$-module $M$ is {\em extended} from $R$ if there exists a finitely generated projective $R$-module $M'$ such that $M\cong R[T]\otimes_RM'$.
 An extension of Serre's Problem proved by Quillen and Suslin is the following

\begin{Conjecture} (Bass-Quillen) \label{bq} If $R$ is a regular ring then every finitely generated projective module over $R[T]$ is extended from $R$.
\end{Conjecture}

If $(R,m)$ is regular local ring then the above conjecture says that every finitely generated projective module over $R[T]$ is free. For a possible proof it is enough to consider only this case  using the Quillen's Patching Theorem \cite[Theorem 1]{Q}.
A consequence of Theorem \ref{gnd} is a partial answer to a question of Swan which was the bases of the following theorem.

\begin{Theorem} (Popescu \cite{P0})\label{p0} The Bass-Quillen Conjecture holds if $R$ is a regular local ring in one of the cases:
\begin{enumerate}
\item{} $R$ contains a field,

\item{} the characteristic $p$ of $k$  is not in $m^2$,

\item{} $R$ is excellent Henselian.
\end{enumerate}
\end{Theorem}

An interesting problem is to replace in the Bass-Quillen Conjecture the polynomial algebra $R[T]$ by other $R$-algebras. A useful result in this idea is the following theorem which we found  too late to use it in \cite{P0}.

\begin{Theorem} (Vorst \cite{V})\label{v} Let $A$ be a ring, $A[X]$, $X=(X_1,\ldots,X_r)$ a polynomial algebra, $I\subset A[X]$ a monomial ideal and $B=A[X]/I$. Then every finitely generated projective $B$-module $M$ is extended  from a finitely generated projective $A$-module if for all $n\in \bf N$ every  finitely generated projective $A[T]$-module, $T=(T_1,\ldots,T_n)$ is extended from a  finitely generated projective $A$-module.
\end{Theorem}

Using the above theorem and Corollary \ref{p0} we get the following theorem.
\begin{Theorem} \label{m1} Let $R$ be a regular local ring in one of the cases  of Theorem \ref{p0}, $I\subset R[X]$, $X=(X_1,\ldots,X_r)$ be a monomial ideal and $B=R[X]/I$. Then any finitely generated projective $B$-module is free.
\end{Theorem}
 The Bass-Quillen Conjecture could also hold when $R$ is not regular as shows the following theorem.
\begin{Theorem}\label{m2} Let $R$ be a regular local ring  in one of the cases  of Theorem \ref{p0},   $I\subset R[X]$, $X=(X_1,\ldots,X_r)$ be a monomial ideal and $A=R[X]/I$. Then every finitely generated $A[T]$-module, $T=(T_1,\ldots, T_n)$ is free.
\end{Theorem}
For the proof note that $A[T]$ is a factor of $R[X,T]$ by the monomial ideal $IR[X,T]$.
\begin{Remark} {\em If  $I$ is not monomial above  then the Bass-Quillen Conjecture may fail for $A$. Indeed, if $A={\bf R}[X_1,X_2]/(X_1^2-X_2^3)$ then there exist finitely generated projective $A[T]$-modules of rank one which are not free (see \cite[(5.10)]{La}).}
\end{Remark}

{\bf Acknowledgment}

The author would like to thank Herwig Hauser for suggesting the statement of  Theorem \ref{pp} and Guillaume Rond for talks on the proof of Theorem \ref{arc}.
 The main part of this paper  was done within the special semester on Artin Approximation of the Chaire Jean Morlet at CIRM, Luminy, spring 2015.
 Also the author owes  thanks to the members of the Luminy seminar organized by Hauser and Rond for  several discussions, and to a Referee who hinted a mistake in a earlier proof of Theorem \ref{arc1}.


\vskip 0.5 cm

\begin{thebibliography}{99}
\bibitem{An} M.\ Andre, {\em Cinq exposes sur la desingularisation}, Handwritten manuscript Ecole Polytechnique Federale de Lausanne, (1991).

\bibitem{A} M.\ Artin, {\em On the solutions of analytic equations}, Invent. Math., {\bf 5}, (1968), 277-291.
\bibitem{A1} M.\ Artin, {\em Algebraic approximation of structures over complete local rings}, Publ. Math. IHES, {\bf 36} (1969), 23-58.
\bibitem{A2} M.\ Artin, {\em Constructions technques for algebraic spaces}, Actes Congres. Intern. Math., t 1, (1970), 419-423-291.
\bibitem{A3} M.\ Artin, {\em Algebraic structure of power series rings}, Contemp. Math. AMS, Providence,1982, 223-227.
\bibitem{AD} M.\ Artin, J.\ Denef, {\em Smoothing of a ring homomorphism along a section}, Arithmetic and Geometry, vol. II, Birkh\"auser, Boston, (1983), 5-32.
\bibitem{CP} M.\ Cipu, D.\ Popescu, {\em  A desingularization theorem of Neron type}, Ann.Univ. Ferrara,
{\bf 30} (1984), 63-76.
\bibitem{Sing} W.\ Decker, G.-M.\ Greuel, G.\ Pfister, H.\ Sch{\"o}nemann: \newblock {\sc Singular} {3-1-6} --- {A} computer algebra system for polynomial  computations.\newblock {http://www.singular.uni-kl.de} (2012).

\bibitem{Gr} M.\ Greenberg, {\em Rational points in henselian discrete valuation rings}, Publ. Math. IHES, {\bf 31}, (1966), 59-64.
\bibitem{K}  H.\ Kurke,  T.\ Mostowski, G.\ Pfister, D.\ Popescu, M.\ Roczen,  {\em Die Approximationseigenschaft
lokaler Ringe}, Springer Lect. Notes in Math., {\bf 634}, Springer-Verlag, Berlin-New
York, (1978).
\bibitem{La} T.\ Y.\ Lam, {\em Serre's Conjecture}, Springer Lect. Notes in Math., {\bf 635}, Berlin, 1978.
\bibitem{LJ} M.\ Lejeune-Jalabert, {Courbes trac\'ees sur un germe  d'hypersurface}, Amer. J. Math., {\bf 112} (1990), 525-568.
\bibitem{N} A.\ N\'eron, {\em Modeles minimaux des varietes abeliennes sur les corps locaux et globaux}, Publ. Math.  IHES, {\bf 21}, 1964.
\bibitem{PP} G.\ Pfister, D.\ Popescu, {\em  Die strenge Approximationseigenschaft lokaler Ringe}, Inventiones Math.
{\bf 30}, (1975),145-174.

\bibitem{Pl} A.\ Ploski, {\em Note on a theorem of M. Artin}, Bull. Acad. Polon. des  Sci., t. XXII, 11 (1974), 1107-1110.
\bibitem{AP} A.\ Popescu, D.\ Popescu, {\em A method to compute the General Neron Desingularization in the frame of one dimensional local domains},  arXiv:AC/1508.05511.
 \bibitem{P3} D.\ Popescu, {\em General Neron Desingularization and approximation}, Nagoya Math. J., {\bf 104} (1986), 85-115.
 \bibitem{P''} D.\ Popescu, {\em Letter to the Editor. General Neron Desingularization and approximation}, Nagoya Math. J., {\bf 118} (1990), 45-53.
\bibitem{P0} D.\ Popescu, {\em Polynomial rings and their projective modules}, Nagoya Math. J., {\bf 113}, (1989), 121-128.

\bibitem{P2} D.\ Popescu, {\em Variations on N\'eron desingularization}, in: Sitzungsberichte der Berliner
Mathematischen Gesselschaft, Berlin, 2001, 143-151.
\bibitem{P1} D.\ Popescu, {\em Artin Approximation}, in "Handbook of Algebra", vol. 2, Ed. M. Hazewinkel, Elsevier, 2000, 321-355.

\bibitem{Q} D.\ Quillen, {\em Projective modules over polynomial rings}, Invent. Math., {\bf 36}, (1976), 167-171.
\bibitem{Sp} M.\ Spivakovski, {\em A new proof of D. Popescu's theorem on smoothing of ring homomorphisms}, J. Amer. Math. Soc., {\bf 294} (1999), 381-444.
\bibitem{S} R.\ Swan, {\em Neron-Popescu desingularization}, in "Algebra and Geometry", Ed. M. Kang, International Press, Cambridge, (1998), 135-192.
\bibitem{V} T.\ Vorst, {\em The Serre Problem for discrete Hodge algebras}, Math. Z., {\bf 184}, (1983), 425-433.

\end{thebibliography}
\end{document}